\documentclass[11pt,reqno]{amsart}
\usepackage[T1]{fontenc}
\usepackage[utf8]{inputenc}
\usepackage{lmodern}
\usepackage[a4paper,margin=28mm,headheight=14pt]{geometry}
\usepackage{amsmath,amssymb,amsthm,mathtools,mathrsfs}
\usepackage[expansion=false]{microtype}
\usepackage{enumitem}
\usepackage{needspace}
\usepackage[hidelinks,unicode]{hyperref}
\numberwithin{equation}{section}
\allowdisplaybreaks[2]
\setlist[enumerate]{label=\textup{(\roman*)},leftmargin=2.1em,itemsep=3pt}
\theoremstyle{plain}
\newtheorem{theorem}{Theorem}[section]
\newtheorem{proposition}[theorem]{Proposition}

\theoremstyle{definition}

\theoremstyle{remark}
\newtheorem{example}[theorem]{Example}
\newcommand{\C}{\mathbb C}
\newcommand{\R}{\mathbb R}

\newcommand{\PP}{\mathbb P}
\newcommand{\OO}{\mathcal O}
\newcommand{\HH}{\mathcal H}
\newcommand{\KK}{\mathcal K}
\newcommand{\DD}{\mathscr D}

\newcommand{\Herm}{\operatorname{Herm}}

\newcommand{\tr}{\operatorname{tr}}

\newcommand{\rk}{\operatorname{rank}}
\newcommand{\conv}{\operatorname{conv}}
\newcommand{\Levi}{\operatorname{Levi}}
\newcommand{\Rea}{\operatorname{Re}}

\newcommand{\ip}[2]{\left\langle #1,#2\right\rangle}
\newcommand{\ddc}{i\partial\bar\partial}
\title[First-jet characterization of Griffiths positivity]{A first-jet characterization of Griffiths positivity}
\author{Yun-Heng Du}
\address{Academy of Mathematics and Systems Science, Chinese Academy of Sciences, Beijing 100190, China}
\email{duyunheng@amss.ac.cn}
\author{Song-Yan Xie}
\address{State Key Laboratory of Mathematical Sciences, Academy of Mathematics and Systems Science, Chinese Academy of Sciences, Beijing 100190, China; School of Mathematical Sciences, University of Chinese Academy of Sciences, Beijing 100049, China}
\email{xiesongyan@amss.ac.cn}
\date{}
\hypersetup{pdftitle={A first-jet characterization of Griffiths positivity},pdfauthor={Yun-Heng Du and Song-Yan Xie},pdfsubject={A first-jet criterion for Griffiths positivity and an explicit example with degenerate untwisted evaluation},pdfkeywords={Griffiths positivity, first jets, Levi form, matrix-valued measures, Steiner bundles}}
\begin{document}
\begin{abstract}
We characterize Griffiths positivity of vector bundles on smooth projective varieties by a finite-dimensional convex condition on first jets of sections of positive twists. A nonzero positive matrix-valued measure annihilated by the adjoint Levi operator gives the dual obstruction. We also give an explicit bundle on an abelian surface that satisfies the first-level condition, but whose complete untwisted evaluation has a differential kernel.
\end{abstract}
\maketitle

\section{Introduction and main result}\label{sec:intro}

Let $E\to X$ be a holomorphic vector bundle of positive rank over a connected smooth complex projective variety. Griffiths~\cite{Griffiths} proved that Griffiths positivity implies ampleness and conjectured the converse. Du--Xie~\cite{DX} show that ampleness alone does not determine the existence of a Griffiths-positive metric. The question is which data of a holomorphic bundle detect Griffiths positivity.

Related recognition questions arise for cotangent bundles of complete intersections. The high-degree ampleness statement proposed by Debarre~\cite{Debarre} was first proved in full generality by Xie~\cite{Xie}, using explicit differential forms of Brotbek~\cite{Brotbek2016}. A subsequent proof was given by Brotbek--Darondeau~\cite{BrotbekDarondeau}. Under a stronger codimension hypothesis, Mohsen~\cite{Mohsen} constructed complete intersections with Griffiths-positive cotangent metrics. In the local model of that construction, positivity is expressed by immersivity of a Gauss map~\cite[Lemma~13]{Mohsen}.

Deng--Ning--Wang--Zhou~\cite[Theorems~1.2 and~1.3]{DNWZ} give criteria for Griffiths semipositivity in terms of $L^p$ estimates for $\bar\partial$ and holomorphic extension conditions. For a fixed smooth Hermitian metric, Liu--Xu~\cite[Theorem~1.1]{LiuXu} characterize Griffiths curvature lower bounds by local $L^2$ extension conditions. Demailly~\cite{Demailly} approaches the existence of positive metrics through a nonlinear elliptic system of Hermitian--Yang--Mills type.

A complementary approach to geometric positivity comes from duality between positive forms and currents, as in Sullivan's~\cite{Sullivan} theory of structure currents and Harvey--Lawson's~\cite[Theorem~14]{HarveyLawson} characterization of K\"ahler manifolds. We use such a duality to study the global existence of a Griffiths-positive metric, by expressing curvature positivity as a linear condition on Hermitian forms.

Ampleness controls the tautological line bundle on a projectivized bundle, while Griffiths positivity is a condition on a single Hermitian quadratic norm on each vector-space fibre; complex Finsler geometry makes this distinction explicit~\cite{KobayashiFinsler}. A second difficulty is that the Chern curvature is nonlinear in the metric coefficients. For a real $C^2$ function $u$ on a complex manifold, its Levi form is the Hermitian form $\Levi(u)(\xi,\eta)=(\partial\bar\partial u)(\xi,\bar\eta)$ on $(1,0)$ tangent vectors. It is strictly positive at a point $p$ if $\Levi(u)_p(\xi,\xi)>0$ for every nonzero $(1,0)$ tangent vector $\xi$ at $p$; $u$ is strictly plurisubharmonic when this holds at every point. Applying this operator to fibrewise Hermitian quadratic functions on the dual bundle turns the metric existence problem into a linear positivity condition.

Write $F\coloneqq E^*$ for the dual bundle. A smooth Hermitian form $h$ on $F$, not assumed positive, determines the fibrewise quadratic function
\[
 G_h(v)\coloneqq h(v,v).
\]
Throughout, $\Levi(G_h)$ is taken on the total space of $F$, so its arguments include both base and fibre directions. It is real linear in $h$. For a positive-definite metric $h$, the condition $\Levi(G_h)_v(\xi,\xi)>0$ for every $v\in F\setminus0$ and every nonzero $\xi\in T_v^{1,0}F$ is equivalent to Griffiths negativity of $h$, hence to Griffiths positivity of the dual metric on $E$. This linearization gives an exact theorem of alternatives: the obstruction is a nonzero positive matrix-valued measure annihilating the Levi forms of all fibrewise Hermitian quadratic functions. After fixing a projective polarization, Griffiths positivity is detected at some finite level by first jets of global sections, whereas failure at every level yields such a measure. The first jet of a holomorphic section at a point records its value and first derivatives in local holomorphic coordinates and a local holomorphic frame; two sections have the same first jet precisely when their difference vanishes to order at least two at that point.

Fix a very ample line bundle $A$ with the Fubini--Study metric $a$ induced by a positive Hermitian inner product on $H^0(X,A)$. Since $A$ is ample, Serre's global generation theorem allows us to choose an integer $m_0\ge0$ such that $E\otimes A^{m_0}$ is globally generated. For $m\ge m_0$ set
\[
 V_m\coloneqq H^0(X,E\otimes A^m),\qquad N_m\coloneqq\dim V_m.
\]
Choose a basis $s_1^{(m)},\ldots,s_{N_m}^{(m)}$ of $V_m$, and let $\Herm(N_m)$ be the real vector space of Hermitian $N_m\times N_m$ matrices. A matrix $B\in\Herm(N_m)$ determines a smooth Hermitian form $h_{m,B}$ on $F$ by
\[
 G_{m,B}(x,v)=h_{m,B,x}(v,v)
 \coloneqq \sum_{i,j=1}^{N_m}B_{ij}\,a_x^m\bigl(v(s_j^{(m)}(x)),v(s_i^{(m)}(x))\bigr).
\]
Here $v\in F_x$, and $a^m$ pairs the evaluated sections in $A_x^m$, linearly in its first argument. The forms $h_{m,B}$ comprise the real vector space $\HH_m$; no positivity assumption is imposed on $B$. This is the construction of Bergman sections for arbitrary Hermitian forms in Lempert~\cite[Sections~3 and~4]{Lempert}, applied to $E\otimes A^m$ and contracted with $a^m$.

Fix a background Hermitian metric $k$ on $F$, with unit sphere bundle $S_kF\coloneqq\{v\in F:k(v,v)=1\}$, and a Hermitian metric $b$ on $W\coloneqq T^{1,0}F|_{S_kF}$. The corresponding unit tangent vectors form the compact space
\begin{equation}\label{eq:K}
 \KK\coloneqq\{(v,\xi):v\in S_kF,\ \xi\in T_v^{1,0}F,\ b(\xi,\xi)=1\}.
\end{equation}
For a general Hermitian form $h$ write
\[
 \DD h\coloneqq\Levi(G_h)|_{S_kF}.
\]
A positive matrix-valued Radon measure $T$ on $S_kF$ acts on continuous Hermitian forms on $W$ and is nonnegative on pointwise positive semidefinite forms. The distribution $\DD^*T$ on $X$ acts on smooth Hermitian forms on $F$ by
\[
 \ip{\DD^*T}{h}=\ip{T}{\DD h}.
\]

Our main theorem characterizes Griffiths positivity both by a finite-level first-jet condition and by the absence of a positive measure obstruction.

\begin{theorem}\label{thm:A}
For every $m\ge m_0$ there is a continuous map
\[
 \Psi_m:\KK\longrightarrow\Herm(N_m)
\]
from the compact space $\KK$ in \eqref{eq:K} to the finite-dimensional real vector space $\Herm(N_m)$, satisfying
\begin{equation}\label{eq:jetintro}
 \Levi(G_{m,B})_v(\xi,\xi)=\tr(B\Psi_m(v,\xi)).
\end{equation}
The following conditions are equivalent:
\begin{enumerate}
\item $E$ admits a smooth Griffiths-positive Hermitian metric;
\item for some $m\ge m_0$ and some $B\in\Herm(N_m)$,
\[
 \tr(B\Psi_m(q))>0\qquad(q\in\KK);
\]
\item for some $m\ge m_0$,
\[
 0\notin\conv\Psi_m(\KK),
\]
where $\conv$ denotes the convex hull;
\item there is no nonzero positive matrix-valued Radon measure $T$ with $\DD^*T=0$.
\end{enumerate}
\end{theorem}

Level $m$ succeeds if the inequality in \textup{(ii)} holds for some $B$, equivalently if this trace pairing has a positive minimum on $\KK$. Some finite level succeeds precisely when $E$ is Griffiths positive, so this existence condition depends only on $E$; otherwise every level fails. When level $m$ fails, Proposition~\ref{prop:finitealternative} gives a probability measure $\mu_m$ supported on at most $N_m^2+1$ points of $\KK$ and satisfying $\int_\KK\Psi_m\,d\mu_m=0$.

The proof begins by expressing the Levi form at each level through first jets of global sections of $E\otimes A^m$ and of the polarization $A$. Allowing arbitrary Hermitian coefficient matrices makes $\HH_m$ a real linear space, so finite-dimensional separation gives either a strictly positive trace pairing or an atomic probability measure annihilating that level.

The passage from finite levels to all smooth metrics rests on nestedness and $C^2$-density. The Fubini--Study identity makes the spaces $\HH_m$ nested, while the Bergman kernel expansion in Wang~\cite[Appendix, Theorems~5.1 and~5.2]{Wang} supplies the approximation. For a Griffiths-negative metric $h$ on $F$, the inequality $\DD h>0$ on the compact sphere bundle $S_kF$ persists under sufficiently close $C^2$ approximation, giving success at a finite level. If every level fails, a weak limit of the atomic probability measures annihilates all finite levels by nestedness, and then all smooth Hermitian forms by density. This limit induces a nonzero positive matrix-valued measure $T$ satisfying $\DD^*T=0$.

The first-jet identity \eqref{eq:jetintro} gives a geometric interpretation of the untwisted level: for a globally generated bundle $E$, level zero succeeds precisely when the complete evaluation map
\[
 E^*\setminus0\longrightarrow H^0(X,E)^*
\]
is an immersion. This is close to the Gauss-map criterion in Mohsen's complete-intersection construction~\cite[Lemma~13]{Mohsen}. Twisting can recover tangent directions lost by the untwisted evaluation map, as the example below shows.

We apply the criterion to an explicit pullback $E\coloneqq f^*S$ of a Steiner bundle, of the type used as a comparison bundle in Du--Xie~\cite{DX}. This pullback is different from the counterexample in that paper: the counterexample was obtained later by deformation and a moduli-space argument, whereas the present pullback is Griffiths positive. In Section~\ref{sec:examples} we show that ramification obstructs level zero and construct a first-level matrix satisfying the first-jet inequality of Theorem~\ref{thm:A}. This gives an explicit case in which twisting detects positivity despite the differential kernel of the complete untwisted evaluation.

\Needspace{12\baselineskip}
\begin{example}\label{ex:levelone}
For the elliptic curve $C:y^2=x^3-x$ with point at infinity $\infty$, consider
\[
 X\coloneqq C\times C,\qquad H\coloneqq \OO_C(4[\infty])\boxtimes\OO_C(5[\infty]).
\]
Here $L_1\boxtimes L_2=\operatorname{pr}_1^*L_1\otimes\operatorname{pr}_2^*L_2$ denotes the external tensor product, where $\operatorname{pr}_i:X\to C$ are the two projections. There is an explicit rank-two bundle $E$ on $X$ with the following properties for the polarization $H$ endowed with a suitable Fubini--Study metric:
\begin{enumerate}[ref=\roman*]
\item\label{ex:zerofailure} The complete level-zero evaluation has a nonzero differential kernel, so level zero fails.
\item\label{ex:positivelevel} At level one the Hermitian matrix $B_1$ in \eqref{eq:examplematrix} satisfies
\[
 \tr(B_1\Psi_1(q))>0\qquad(q\in\KK),
\]
so level one succeeds.
\end{enumerate}
Theorem~\ref{thm:A} applied to \textup{(\ref{ex:positivelevel})} gives Griffiths positivity of $E$. Together with \textup{(\ref{ex:zerofailure})}, this shows how twisting by $H$ detects positivity despite the differential kernel of the complete untwisted evaluation.
\end{example}

\noindent\textbf{Plan of the paper.} Section~\ref{sec:levi} develops the Levi operator and proves Theorem~\ref{thm:A} through the first-jet formula, the finite-dimensional alternative, and $C^2$-density. In Section~\ref{sec:examples}, we verify Example~\ref{ex:levelone}: first the failure at level zero, then the construction of a first-level matrix and verification of the strict inequality.

\section{Finite-dimensional reduction of Griffiths positivity}\label{sec:levi}

For a real $C^2$ function $u$, we write $\Levi(u)(\xi,\eta)=(\partial\bar\partial u)(\xi,\bar\eta)$ for its Levi form on $(1,0)$ tangent vectors. Let $h$ be a smooth Hermitian form on $F=E^*$, not assumed positive definite, with fibrewise quadratic function $G_h(v)=h(v,v)$. At a point $v\in F_x\setminus0$, a vector $w\in F_x$ represents a vertical tangent vector through the holomorphic curve $t\mapsto v+tw$. Restricting $G_h$ to this curve gives
\begin{equation}\label{eq:verticallevi}
 \Levi(G_h)_v(w,w)=\left.\frac{\partial^2}{\partial t\partial\bar t}h_x(v+tw,v+tw)\right|_{t=0}=h_x(w,w).
\end{equation}
If $G_h$ is strictly plurisubharmonic on $F\setminus0$, the left-hand side is positive for every $w\ne0$. Since $x$ is arbitrary, $h$ is then a positive-definite metric on $F$. For a smooth positive-definite Hermitian metric $h$ on $F$, Griffiths negativity of $h$ is equivalent to strict plurisubharmonicity of $G_h$ on $F\setminus0$, by Drinovec Drnov\v sek--Forstneri\v c~\cite[Proposition~6.2(ii)--(iii)]{DDF}.

In a local holomorphic frame $e_A$ of $A$, write $|e_A|_a^2=e^{-\phi}$ for the fixed metric $a$. Write $V_m=H^0(X,E\otimes A^m)$ and $N_m=\dim V_m$, and choose a basis $s_1^{(m)},\ldots,s_{N_m}^{(m)}$ of $V_m$. Vectors are written as columns. For $v\in F_x$, let $p_m(x,v)\in\C^{N_m}$ be the column determined by
\begin{equation}\label{eq:pmdefinition}
 v(s_j^{(m)}(x))=(p_m(x,v))_j e_A(x)^m,
 \qquad 1\le j\le N_m.
\end{equation}
The map $p_m$ is holomorphic on the total space of $F$ over the trivializing open set. Let $\Herm(N_m)$ be the real vector space of Hermitian $N_m\times N_m$ matrices, and for $B\in\Herm(N_m)$ let $h_{m,B}$ denote the smooth Hermitian form on $F$ with local expression
\begin{equation}\label{eq:finiteform}
 G_{m,B}(x,v)=h_{m,B,x}(v,v)=e^{-m\phi(x)}p_m(x,v)^{\dagger}B p_m(x,v).
\end{equation}
Here $\dagger$ denotes conjugate transpose. Under a change of local frame $e_A'=\gamma e_A$, with $\gamma$ holomorphic and nowhere vanishing,
\[
 p_m'=\gamma^{-m}p_m,\qquad \phi'=\phi-\log|\gamma|^2,
\]
so \eqref{eq:finiteform} gives the same Hermitian form on overlaps. A change of basis of $V_m$ replaces $p_m$ by $Cp_m$ for a constant invertible matrix $C$, and
\[
 (Cp_m)^{\dagger}B(Cp_m)=p_m^{\dagger}(C^{\dagger}BC)p_m.
\]
Since $B\mapsto C^{\dagger}BC$ is a bijection of $\Herm(N_m)$, the space $\HH_m=\{h_{m,B}:B\in\Herm(N_m)\}$ is independent of the chosen basis.

For $q=(v,\xi)\in\KK$, write $\xi=(u,w)$ in a local holomorphic trivialization, with base component $u$ and fibre component $w$. For $u=\sum_i u_i\partial/\partial z_i$, the derivatives of the local weight are $\phi_u=\sum_i u_i\partial_i\phi$ and $\phi_{u\bar u}=\sum_{i,j}u_i\bar u_j\partial_i\partial_{\bar j}\phi$. Set
\[
 r_m\coloneqq dp_m(\xi)-m\phi_u p_m,\qquad \kappa\coloneqq\phi_{u\bar u}.
\]
Differentiating \eqref{eq:finiteform} gives
\begin{align*}
 \Levi(G_{m,B})(\xi,\xi)
 &=e^{-m\phi}\bigl(dp_m(\xi)^{\dagger}B\,dp_m(\xi)-2m\Rea(\phi_u\,dp_m(\xi)^{\dagger}Bp_m)
 \\ &\hspace{4em} +(m^2|\phi_u|^2-m\kappa)p_m^{\dagger}Bp_m\bigr)\\
 &=e^{-m\phi}(r_m^{\dagger}Br_m-m\kappa p_m^{\dagger}Bp_m).
\end{align*}
This calculation leads to the local matrix-valued map
\begin{equation}\label{eq:Psiformula}
 \Psi_m\coloneqq e^{-m\phi}\bigl(r_mr_m^{\dagger}-m\kappa p_mp_m^{\dagger}\bigr),
\end{equation}
for which $\Levi(G_{m,B})_v(\xi,\xi)=\tr(B\Psi_m(v,\xi))$ for every Hermitian matrix $B$. Under the frame change above, $r_m'=\gamma^{-m}r_m$ and $\kappa'=\kappa$, so the expression in \eqref{eq:Psiformula} is unchanged. Thus, for the chosen basis of $V_m$, these local expressions yield a continuous global map $\Psi_m:\KK\to\Herm(N_m)$ satisfying \eqref{eq:jetintro}, as required in Theorem~\ref{thm:A}.

To express $\phi_u$ and $\kappa$ in \eqref{eq:Psiformula} through first jets, choose an orthonormal basis $t_0,\ldots,t_N$ for the fixed inner product on $H^0(X,A)$. Since $a$ is the induced Fubini--Study metric,
\begin{equation}\label{eq:FSmetric}
 e^\phi=\sum_{\alpha=0}^N|t_\alpha|^2,
\end{equation}
where $t_\alpha$ also denotes the coefficient of the section in the frame $e_A$. Using superscript $t$ for transpose, write $t\coloneqq(t_0,\ldots,t_N)^t$ and $\rho\coloneqq t^{\dagger}t$. Differentiating $\phi=\log\rho$ gives
\[
 \phi_u=\frac{t^{\dagger}dt(u)}{\rho},\qquad
 \kappa=\phi_{u\bar u}=\frac{\rho\,|dt(u)|^2-|t^{\dagger}dt(u)|^2}{\rho^2}.
\]
Thus \eqref{eq:Psiformula} depends only on first jets of sections of $E\otimes A^m$ and of the fixed polarization $A$.

\begin{proposition}\label{prop:finitealternative}
For a fixed $m\ge m_0$, a Hermitian matrix $B$ with $\tr(B\Psi_m)>0$ on $\KK$ exists if and only if $0\notin\conv\Psi_m(\KK)$. If $0\in\conv\Psi_m(\KK)$, there are points $q_1,\ldots,q_s\in\KK$ and weights $\lambda_j\ge0$ such that
\begin{equation}\label{eq:finiteannintro}
 s\le N_m^2+1,\qquad \sum_{j=1}^s\lambda_j=1,
 \qquad \sum_{j=1}^s\lambda_j\Psi_m(q_j)=0.
\end{equation}
\end{proposition}
\begin{proof}
Let $C_m\coloneqq\conv\Psi_m(\KK)$. Carath\'eodory's theorem, applied to $\Psi_m(\KK)$ in the real vector space $\Herm(N_m)$ of dimension $N_m^2$, shows that every $Z\in C_m$ can be written as
\[
 Z=\sum_{j=1}^{N_m^2+1}\lambda_j\Psi_m(q_j),\qquad
 q_j\in\KK,\quad\lambda_j\ge0,\quad\sum_{j=1}^{N_m^2+1}\lambda_j=1.
\]
Thus the continuous map $((q_j),(\lambda_j))\mapsto\sum_j\lambda_j\Psi_m(q_j)$ maps the compact space $\KK^{N_m^2+1}$ times the closed weight simplex onto $C_m$, proving that $C_m$ is compact. If $0\in C_m$, taking $Z=0$ in this representation and discarding zero weights gives \eqref{eq:finiteannintro}. No Hermitian $B$ can then satisfy $\tr(B\Psi_m(q))>0$ for every $q\in\KK$, since the weighted sum of these traces is zero. If $0\notin C_m$, the strict separation theorem gives a real-linear functional $\ell$ on $\Herm(N_m)$ and a constant $c>0$ such that $\ell(Z)\ge c$ for every $Z\in C_m$, while $\ell(0)=0$. The pairing $(B,Z)\mapsto\tr(BZ)$ is a real inner product on $\Herm(N_m)$. Choose an orthonormal basis $E_1,\ldots,E_{N_m^2}$ for this inner product and set
\[
 B\coloneqq\sum_{a=1}^{N_m^2}\ell(E_a)E_a\in\Herm(N_m).
\]
For $Z=\sum_a z_aE_a$ with $z_a\in\R$, orthonormality and real linearity give
\[
 \tr(BZ)=\sum_a\ell(E_a)z_a=\ell\!\left(\sum_a z_aE_a\right)=\ell(Z).
\]
Hence $\tr(BZ)\ge c>0$ for every $Z\in C_m$.
\end{proof}

On the compact space $\KK=\{(v,\xi):v\in S_kF,\ \xi\in T_v^{1,0}F,\ b(\xi,\xi)=1\}$ of unit tangent vectors over $S_kF=\{v\in F:k(v,v)=1\}$, with the fixed Hermitian metrics $k$ on $F$ and $b$ on $W=T^{1,0}F|_{S_kF}$, set
\[
 f_h(v,\xi)\coloneqq\Levi(G_h)_v(\xi,\xi).
\]

Let $m\ge m_0$, where $m_0$ is the integer chosen so that $E\otimes A^{m_0}$ is globally generated. Our choice of the Fubini--Study metric gives $\sum_{\alpha=0}^N e^{-\phi}|t_\alpha|^2=1$ by \eqref{eq:FSmetric}, hence
\[
 G_{m,B}=e^{-(m+1)\phi}\sum_{\alpha=0}^N(p_mt_\alpha)^{\dagger}B(p_mt_\alpha).
\]
Express the product sections in the chosen basis of $V_{m+1}$:
\[
 s_j^{(m)}t_\alpha=\sum_{\ell=1}^{N_{m+1}}(C_\alpha)_{j\ell}s_\ell^{(m+1)},
 \qquad C_\alpha\in\operatorname{Mat}_{N_m\times N_{m+1}}(\C).
\]
By \eqref{eq:pmdefinition}, applying $v\in F_x=E_x^*$ to the $E_x$ components of this identity and taking coefficients in $e_A^{m+1}$ gives $p_m(x,v)t_\alpha(x)=C_\alpha p_{m+1}(x,v)$, where $t_\alpha(x)$ denotes the local coefficient of the section. Substitution into the preceding identity yields
\[
 G_{m,B}=e^{-(m+1)\phi}p_{m+1}^{\dagger}
 \left(\sum_{\alpha=0}^N C_\alpha^{\dagger}BC_\alpha\right)p_{m+1}
 =G_{m+1,B'},\qquad
 B'\coloneqq\sum_{\alpha=0}^N C_\alpha^{\dagger}BC_\alpha.
\]
Since $B'$ is Hermitian, $h_{m,B}=h_{m+1,B'}\in\HH_{m+1}$ for every $B\in\Herm(N_m)$. Therefore $\HH_m\subset\HH_{m+1}$ for every $m\ge m_0$.

We will use approximation of metrics in the $C^2$ topology. In a local holomorphic frame, let $\mathsf H_h$ denote the coefficient matrix of $h$, so that $h(v,w)=w^{\dagger}\mathsf H_h v$ for the coordinate columns $v,w$. For a tangent vector $(u,w)$ in total-space coordinates $(z,v)$, differentiation of $G_h(v)=v^{\dagger}\mathsf H_h v$ gives
\begin{equation}\label{eq:LeviRaw}
 \Levi(G_h)(u,w)
 =w^{\dagger}\mathsf H_h w+2\Rea(w^{\dagger}(\mathsf H_h)_u v)+v^{\dagger}(\mathsf H_h)_{u\bar u}v,
\end{equation}
where $(\mathsf H_h)_u\coloneqq\sum_i u_i\partial_i\mathsf H_h$ and $(\mathsf H_h)_{u\bar u}\coloneqq\sum_{i,j}u_i\bar u_j\partial_i\partial_{\bar j}\mathsf H_h$. Since $X$ is projective and hence compact, choose finitely many coordinate neighborhoods $\widetilde U_\nu$, each carrying a holomorphic frame of $F$, and open sets $U_\nu\Subset\widetilde U_\nu$ that cover $X$. If $\mathsf H_h^{(\nu)}$ is the coefficient matrix in the frame on $\widetilde U_\nu$, use the norm
\[
 \|h\|_{C^2}\coloneqq
 \max_{\nu,a,b,\,|\alpha|\le2}\sup_{x\in\overline U_\nu}
 |\partial^\alpha(\mathsf H_h^{(\nu)})_{ab}(x)|,
\]
where $\partial^\alpha$ denotes differentiation in the underlying real coordinates. For a continuous Hermitian form $\Phi$ on $W$, use
\[
 \|\Phi\|_{C^0}\coloneqq\sup_{(v,\xi)\in\KK}|\Phi_v(\xi,\xi)|,
\]
the supremum of its operator norm with respect to $b$. Writing $\Herm F$ and $\Herm W$ for the real vector bundles of Hermitian forms on $F$ and $W$, respectively, \eqref{eq:LeviRaw} gives a continuous real-linear operator
\[
 \DD:C^2(X,\Herm F)\longrightarrow C^0(S_kF,\Herm W),
 \qquad \|\DD h\|_{C^0}\le C\|h\|_{C^2}.
\]
The constant $C$ depends only on the fixed metrics, coordinate cover, and frames. For scalar functions on $\KK$, write $\|f\|_\infty\coloneqq\sup_{q\in\KK}|f(q)|$. Since $\|f_h\|_\infty=\|\DD h\|_{C^0}$, $C^2$ approximation of $h$ gives uniform approximation of $f_h$ on $\KK$.

The $C^2$-density follows from the standard approximation of Hermitian metrics by Fubini--Study metrics, stated in Hashimoto--Keller~\cite[Section~2.3, p.~11, following (2.9)]{HashimotoKeller}. Applied to the dual $h^*$ of a smooth positive-definite metric $h$ on $F$, this gives Fubini--Study metrics $g_m$ on $E$ with $g_m\to h^*$ in $C^2$. Their duals $g_m^*$ belong to $\HH_m$: they are induced by positive-definite Hermitian forms on $V_m^*$ through the dual evaluation map, with the twist removed by $a^m$. Consequently $g_m^*\to h$ in $C^2$. Every smooth Hermitian form $u$ is a difference $(u+Ck)-Ck$ of positive-definite metrics for a sufficiently large constant $C$. Approximating both at the same levels and subtracting, then using smooth approximation of $C^2$ forms, shows that $\bigcup_{m\ge m_0}\HH_m$ is dense in the space of $C^2$ Hermitian forms on $F$. If $\dim X=0$, this is immediate since every Hermitian form belongs to each $\HH_m$.

\begin{proof}[Proof of Theorem~\ref{thm:A}]
Proposition~\ref{prop:finitealternative} proves the equivalence of \textup{(ii)} and \textup{(iii)}. If \textup{(ii)} holds, then $f_{h_{m,B}}(q)=\tr(B\Psi_m(q))>0$ for every $q\in\KK$. Normalizing the fibre point and tangent vector shows that $G_{h_{m,B}}$ is strictly plurisubharmonic on $F\setminus0$. Equation~\eqref{eq:verticallevi} gives $h_{m,B}>0$. Since $G_{h_{m,B}}$ is strictly plurisubharmonic on $F\setminus0$, Drinovec Drnov\v sek--Forstneri\v c~\cite[Proposition~6.2(ii)]{DDF} implies that $h_{m,B}$ is Griffiths negative. Its dual metric on $E$ is Griffiths positive, proving \textup{(ii)}$\Rightarrow$\textup{(i)}.

Conversely, let $h$ be the dual of a Griffiths-positive metric. Since $G_h$ is strictly plurisubharmonic off the zero section, $f_h$ has a positive minimum $\delta$ on the compact space $\KK$. The $C^2$-density established above, together with the continuity following \eqref{eq:LeviRaw}, gives $h_m\in\HH_m$ for some $m$ such that
\[
 \|f_{h_m}-f_h\|_\infty<\delta/2.
\]
Thus $f_{h_m}>0$, proving \textup{(i)}$\Rightarrow$\textup{(ii)}.

For the dual assertion, we use the representation of positive functionals by matrix-valued measures. The metric $b$ identifies Hermitian forms on $W$ with $b$-self-adjoint endomorphisms. Their real vector bundle is denoted by $\Herm_b(W)$, and its identity section $I_W$ corresponds to $b$. A positive real-linear functional $T$ on $C^0(S_kF,\Herm_b(W))$ is automatically continuous, since
\begin{equation}\label{eq:measurebound}
 |\ip{T}{\Phi}|\le\|\Phi\|_{C^0}\ip{T}{I_W}.
\end{equation}
Here $\Phi$ is a continuous Hermitian form on $W$, with the $C^0$ norm specified above. The Riesz representation theorem in local unitary frames identifies these functionals with positive matrix-valued Radon measures. In particular, $T\ne0$ implies $\ip{T}{I_W}>0$, and $\Phi\ge\delta I_W$ with $\delta>0$ implies $\ip{T}{\Phi}>0$ for every such nonzero $T$.

If \textup{(i)} holds, the dual metric $h$ is Griffiths negative, so $\DD h>0$. Compactness then gives $\DD h\ge\delta I_W$ for some $\delta>0$. For a nonzero positive matrix-valued measure $T$, \eqref{eq:measurebound} therefore gives
\[
 \ip{T}{\DD h}\ge\delta\ip{T}{I_W}>0.
\]
Such a $T$ cannot satisfy $\DD^*T=0$, proving \textup{(i)}$\Rightarrow$\textup{(iv)}.

Suppose now that every level fails. By Proposition~\ref{prop:finitealternative}, $0\in\conv\Psi_m(\KK)$ for every $m\ge m_0$. At each level choose points $q_j$ and weights $\lambda_j$ satisfying \eqref{eq:finiteannintro}, and let $\mu_m\coloneqq\sum_{j=1}^s\lambda_j\delta_{q_j}$, where $\delta_{q_j}$ is the unit point mass at $q_j$. This is a probability measure on $\KK$. For every $h=h_{m,B}\in\HH_m=\{h_{m,B}:B\in\Herm(N_m)\}$, the first-jet identity gives
\[
 \int_\KK f_h\,d\mu_m
 =\sum_{j=1}^s\lambda_j\tr(B\Psi_m(q_j))
 =\tr\!\left(B\sum_{j=1}^s\lambda_j\Psi_m(q_j)\right)=0.
\]
Since $\KK$ is compact and metrizable, a subsequence $\mu_{m_j}$ converges weakly to a probability measure $\mu$. For a fixed level $m'\ge m_0$ and $h\in\HH_{m'}$, the inclusion $\HH_{m'}\subset\HH_m$ for $m\ge m'$ gives
\[
 \int_\KK f_h\,d\mu_m=0\qquad(m\ge m').
\]
Passing to the limit shows that $\mu$ annihilates every finite level. The $C^2$-density and the $C^2$-to-$C^0$ continuity of $\DD$ extend this identity to all smooth Hermitian forms.

The formula
\[
 \ip{T}{\Phi}\coloneqq\int_\KK\Phi_v(\xi,\xi)\,d\mu(v,\xi)
\]
gives a positive matrix-valued Radon measure on $S_kF$. Since $b(\xi,\xi)=1$, its trace mass is $\ip{T}{I_W}=1$, and
\[
 \ip{\DD^*T}{h}=\ip{T}{\DD h}=\int_\KK f_h\,d\mu=0.
\]
Thus any sequence of the chosen atomic measures has a subsequence yielding a nonzero $T$ with $\DD^*T=0$. If \textup{(i)} fails, every level fails because \textup{(ii)} implies \textup{(i)}. The resulting nonzero $T$ contradicts \textup{(iv)}. This proves \textup{(iv)}$\Rightarrow$\textup{(i)} and completes the equivalence of \textup{(i)}--\textup{(iv)}.
\end{proof}

\section{An application on an abelian surface}\label{sec:examples}

\subsection{Construction and failure at level zero}

The obstruction at level zero is detected by the differential of the complete evaluation map. For a globally generated bundle $E$, take $m_0=0$ and write $V\coloneqq V_0=H^0(X,E)$, with $F=E^*$ as before. The evaluation morphism is the holomorphic map
\[
 \widehat\Phi:F\longrightarrow V^*,\qquad
 \widehat\Phi(x,v)(s)\coloneqq v(s(x)).
\]
Equip $V^*$ with a constant Hermitian inner product and use an orthonormal basis. The first-jet identity \eqref{eq:jetintro} at $m=0$ gives
\[
 \Levi(G_{0,B})(\xi,\xi)=d\widehat\Phi(\xi)^{\dagger}B\,d\widehat\Phi(\xi).
\]
Thus the complete level-zero test succeeds precisely when $\widehat\Phi$ is an immersion on $F\setminus0$: injectivity of its differential gives strict positivity for $B=I$, whereas a nonzero kernel vector makes the displayed expression vanish for every Hermitian $B$.

The rows and columns of the Steiner matrix $\mathsf T$ below are indexed starting at zero. On $\PP^2$ with coordinates $[x:y:z]$, consider
\begin{equation}\label{eq:Texplicit}
 \mathsf T(x,y,z)\coloneqq
 \begin{pmatrix}
 x&0&0&z&0\\
 0&x&0&0&z\\
 y&0&x&0&0\\
 0&y&0&x&0\\
 -z&0&y&0&x\\
 0&z&0&y&0\\
 0&0&z&0&y
 \end{pmatrix}.
\end{equation}
The cokernel $S$ occurs in the sequence
\begin{equation}\label{eq:Steiner}
 0\longrightarrow\OO_{\PP^2}(-1)^5
 \xrightarrow{\ \mathsf T\ }\OO_{\PP^2}^{7}
 \longrightarrow S\longrightarrow0.
\end{equation}
This is the type of Steiner quotient used in Du--Xie~\cite{DX}; we prove the properties needed below directly from the displayed matrix.

An elementary row calculation in \eqref{eq:Texplicit} gives $\rk \mathsf T(x,y,z)=5$ at every point of $\PP^2$, so $S$ is locally free of rank two.

For $v=(v_0,\ldots,v_6)^t\in\C^7$, set
\[
 M_v\coloneqq
 \begin{pmatrix}
 v_0&v_2&-v_4\\
 v_1&v_3&v_5\\
 v_2&v_4&v_6\\
 v_3&v_5&v_0\\
 v_4&v_6&v_1
 \end{pmatrix}.
\]
The identity
\[
 \mathsf T(x,y,z)^tv=M_v(x,y,z)^t
\]
is immediate entry by entry.

For a nonzero homogeneous coordinate column $Z\coloneqq(x,y,z)^t\in\C^3$, the dual fibres are
\[
 S^*_{[Z]}=\ker \mathsf T(Z)^t=\{v\in\C^7:M_vZ=0\}.
\]

Let $h_S$ be the metric on $S^*$ induced by the standard metric on $\C^7$, and $g_S\coloneqq h_S^*$ its dual. The holomorphic map $S^*\setminus0\to\C^7$, $([Z],v)\mapsto v$, is an immersion: a tangent vector in its differential kernel satisfies $\dot v=0$ and $M_v\dot Z=0$, while $\rk M_v\ge2$ for $v\ne0$ and $M_vZ=0$ give $\ker M_v=\C Z$, so the base tangent vector also vanishes. Thus $G_{h_S}=|v|^2$ is strictly plurisubharmonic, and the curvature--Levi correspondence used in Section~\ref{sec:levi} shows that $g_S$ is Griffiths positive.

For a Hermitian metric $h$ with local coefficient matrix $\mathsf H_h$, we use the convention in which the local matrix of its Chern curvature $\Theta_h$ is $i\sum R^h_{i\bar j}\,dz_i\wedge d\bar z_j$, where $R^h_{i\bar j}\coloneqq-\partial_{\bar j}(\mathsf H_h^{-1}\partial_i\mathsf H_h)$, as in Kobayashi~\cite[Chapter~I, Section~4]{Kobayashi}. Write $\omega_{\mathrm{FS}}=\ddc\log|Z|^2$. Compactness of $\PP^2$ gives a constant $c_S>0$ such that, in the Griffiths sense,
\begin{equation}\label{eq:Scurv}
 \Theta_{g_S}\ge c_S\omega_{\mathrm{FS}}\otimes I_S.
\end{equation}

We now construct the map $f:X\to\PP^2$ and the bundle $E=f^*S$ in Example~\ref{ex:levelone}, and verify the failure at level zero. Let $C$ be the smooth projective completion of $y^2=x^3-x$, with point $\infty$. The meromorphic functions $x,y$ have their only poles at $\infty$, of orders $2,3$. The curve has genus one. By Riemann--Roch and Serre duality \cite[Chapter~IV, Theorem~1.3 and Example~1.3.6]{HartshorneAG}, a positive-degree line bundle $L$ satisfies $h^0(L)=\deg L$ and $H^1(L)=0$.

On the two factors of $X=C\times C$, choose the bases
\[
 (a_0,a_1,a_2,a_3)\coloneqq(1,x_1,y_1,x_1^2),\qquad
 (b_0,b_1,b_2,b_3,b_4)\coloneqq(1,x_2,y_2,x_2^2,x_2y_2).
\]
Their distinct pole orders show linear independence, and the dimension formula shows that they are complete bases for $\OO_C(4[\infty])$ and $\OO_C(5[\infty])$. The twenty products
\[
 t_{ab}\coloneqq a_ab_b,\qquad 0\le a\le3,\quad0\le b\le4,
\]
form a basis of $H^0(X,H)$ for
\[
 H=\OO_C(4[\infty])\boxtimes\OO_C(5[\infty]).
\]
Both factors are very ample by the degree criterion for curves~\cite[Chapter~IV, Corollary~3.2(b)]{HartshorneAG}, since their degrees are at least $2g(C)+1=3$. Their exterior product is very ample by the product embeddings followed by the Segre embedding. In particular, the complete system of the $t_{ab}$ has positive Fubini--Study curvature.

Consider the sections and the associated map
\begin{equation}\label{eq:fsections}
 f_0\coloneqq t_{00},\qquad f_1\coloneqq t_{34},\qquad f_2\coloneqq t_{30}+t_{04},
 \qquad f\coloneqq[f_0:f_1:f_2]:X\longrightarrow\PP^2.
\end{equation}
The pairs $(1,x^2)$ and $(1,xy)$ have no common zero as sections of $\OO_C(4[\infty])$ and $\OO_C(5[\infty])$, respectively: the second section in each pair is nonzero at infinity by its pole order. They give morphisms $\pi_4\coloneqq[1:x^2]$ and $\pi_5\coloneqq[1:xy]$ to $\PP^1$. The three sections in \eqref{eq:fsections} arise by composing $\pi_4\times\pi_5$ with
\[
 ([u_0:u_1],[v_0:v_1])\longmapsto
 [u_0v_0:u_1v_1:u_1v_0+u_0v_1].
\]
The three displayed coordinates have no common zero on $\PP^1\times\PP^1$, so $f$ is a morphism and $f^*\OO_{\PP^2}(1)=H$.

Since the quotient in \eqref{eq:Steiner} is locally free, the pullback $E=f^*S$ has the exact resolution
\begin{equation}\label{eq:pullresolution}
 0\longrightarrow(H^{-1})^5\longrightarrow\OO_X^7
 \longrightarrow E\longrightarrow0.
\end{equation}

On an elliptic curve a negative-degree line bundle has no sections. The K\"unneth formula applied to the two negative-degree factors of $H^{-1}$ therefore gives
\[
 H^0(X,H^{-1})=H^1(X,H^{-1})=0.
\]
The long exact sequence of \eqref{eq:pullresolution} now gives $H^0(X,E)=\C^7$. Denote the images of the seven standard basis sections of $\OO_X^7$ by $\sigma_0,\ldots,\sigma_6$. They are the pullbacks of the seven quotient sections of $S$.

For the differential kernel in Example~\ref{ex:levelone}\textup{(\ref{ex:zerofailure})}, take
\[
 P\coloneqq((0,0),(2,\sqrt6))\in X.
\]
On the first elliptic curve, $y_1$ is a local parameter at $(0,0)$, and differentiating $y_1^2=x_1^3-x_1$ gives $dx_1=0$ there.
Both nonconstant coordinates of $f$ depend on this factor through $x_1^2$, so
\[
 df_P(u)=0,\qquad u\coloneqq\partial/\partial y_1\ne0.
\]
Choose any unit vector $v\in S^*_{f(P)}$.

In the inclusion $E^*\subset X\times\C^7$, the vector $\xi\coloneqq(u,\dot v=0)$ is tangent to the total space: the derivative of its defining equation $\mathsf T(f(x))^tv=0$ vanishes because $df_P(u)=0$. Since these seven quotient sections form a basis of $H^0(X,E)$, the complete level-zero evaluation is precisely $(x,v)\mapsto v$. Its derivative kills $\xi$. With background $k=f^*h_S$ and the tangent metric $b$, the test space $\KK$ is the unit sphere bundle from \eqref{eq:K}. Thus $q_0\coloneqq(v,\xi/\sqrt{b(\xi,\xi)})\in\KK$ satisfies
\begin{equation}\label{eq:zeroatom}
 \Psi_0(q_0)=0.
\end{equation}
This excludes \emph{every} Hermitian matrix at level zero and gives an atomic obstruction supported at one point. This proves Example~\ref{ex:levelone}\textup{(\ref{ex:zerofailure})}. We next construct a first-level form and its coefficient matrix to prove \textup{(\ref{ex:positivelevel})}.

\subsection{Success at level one}
For the sections $f_0,f_1,f_2$ in \eqref{eq:fsections} and the product basis $(t_{ab})$ of $H^0(X,H)$, use their coefficients in a common local frame $e_H$ of $H$ to write
\[
 r\coloneqq|f_0|^2+|f_1|^2+|f_2|^2,
 \qquad s\coloneqq\sum_{a=0}^3\sum_{b=0}^4|t_{ab}|^2.
\]
Neither expression vanishes, and $s/r$ is a positive smooth function on $X$. The forms
\[
 \omega_r\coloneqq\ddc\log r=f^*\omega_{\mathrm{FS}}\ge0,
 \qquad \omega_s\coloneqq\ddc\log s>0
\]
are global; the strict positivity of $\omega_s$ follows from the very ampleness of $H$. For $M\coloneqq\max_X(s/r)$ and the constant $c_S$ in \eqref{eq:Scurv}, choose
\[
 0<\varepsilon M<\min\{c_S,\tfrac12\}.
\]
Let $a_\varepsilon$ be the metric on $H$ specified by $|e_H|_{a_\varepsilon}^2=(r+\varepsilon s)^{-1}$, and set
\[
 h_1\coloneqq\frac{r+2\varepsilon s}{r+\varepsilon s}\,f^*h_S,
 \qquad g_1\coloneqq h_1^*.
\]
Here $h_1$ is a metric on $F=E^*=f^*S^*$, and $g_1$ is its dual on $E$.

Index the product basis of $H^0(X,H)$ by $j\coloneqq5a+b$, so $0\le j\le19$, and write $t_j\coloneqq t_{ab}$. For the standard coordinate columns $e_j$ of $\R^{20}$, the matrix
\[
 R\coloneqq e_0e_0^t+e_{19}e_{19}^t+(e_{15}+e_4)(e_{15}+e_4)^t
\]
satisfies $r=t^{\dagger}Rt$ and $s=t^{\dagger}t$, where $t\coloneqq(t_0,\ldots,t_{19})^t$. Since $R+\varepsilon I_{20}>0$, the section inner product with Gram matrix $(R+\varepsilon I_{20})^{-1}$ makes $a_\varepsilon$ a Fubini--Study metric as in \eqref{eq:FSmetric}. From now on the spaces $\HH_m$ and maps $\Psi_m$ are taken with $A=H$ and $a=a_\varepsilon$, so $e^\phi=r+\varepsilon s$ in \eqref{eq:finiteform}.

To check the curvature of $g_1$, write $\eta\coloneqq\log(s/r)$ and $\tau\coloneqq\varepsilon e^\eta$. The function $\chi(\eta)\coloneqq\log(1+2\varepsilon e^\eta)-\log(1+\varepsilon e^\eta)$ has derivatives
\[
 \alpha(\tau)\coloneqq\chi'(\eta)=\frac{\tau}{(1+\tau)(1+2\tau)},\qquad
 \beta(\tau)\coloneqq\chi''(\eta)=\frac{\tau(1-2\tau^2)}{(1+\tau)^2(1+2\tau)^2}.
\]
Our choice of $\varepsilon$ gives $0<\alpha(\tau)<c_S$ and $\beta(\tau)>0$. Since $g_1=e^{-\chi(\eta)}f^*g_S$, the Chern curvature formula and \eqref{eq:Scurv} yield
\begin{align*}
 \Theta_{g_1}&=f^*\Theta_{g_S}+\bigl[\alpha(\tau)(\omega_s-\omega_r)+\beta(\tau)i\partial\eta\wedge\bar\partial\eta\bigr]\otimes I_E\\
 &\ge(c_S-\alpha(\tau))\omega_r\otimes I_E+\alpha(\tau)\omega_s\otimes I_E>0.
\end{align*}
Thus $h_1$ is Griffiths negative, including in the directions killed by $df$. We now express it by a finite-level coefficient matrix.

At level one choose a basis $s_1^{(1)},\ldots,s_{N_1}^{(1)}$ of $V_1=H^0(X,E\otimes H)$, where $N_1=\dim V_1$, and write the product sections in this basis:
\[
 \sigma_i t_j=\sum_{\alpha=1}^{N_1}c_{ij,\alpha}s_\alpha^{(1)},
 \qquad 0\le i\le6,\quad0\le j\le19.
\]
Let $\mathsf C\coloneqq(c_{ij,\alpha})$, with rows ordered first by $i$ and then by $j$. For $v\in F_x$, the evaluation column $p_1(x,v)$ records the coefficients of $v(s_\alpha^{(1)}(x))$ in the chosen local frame of $H$. The column of product evaluations is therefore $\mathsf C p_1$. Consider
\begin{equation}\label{eq:examplematrix}
 B_1\coloneqq\mathsf C^{\dagger}\bigl[I_7\otimes(R+2\varepsilon I_{20})\bigr]\mathsf C.
\end{equation}
Since $h_S$ is induced by the standard metric on $\C^7$, one has $G_{f^*h_S}(v)=\sum_{i=0}^6|v(\sigma_i)|^2$. Thus $B_1$ is Hermitian and
\[
 p_1^{\dagger}B_1p_1=\sum_{i=0}^6|v(\sigma_i)|^2\,t^{\dagger}(R+2\varepsilon I_{20})t=(r+2\varepsilon s)\,G_{f^*h_S}(v).
\]
Since $e^\phi=r+\varepsilon s$ for $a_\varepsilon$, division by this denominator gives $G_{1,B_1}=G_{h_1}$ in the notation of \eqref{eq:finiteform}. This also shows that $h_1\in\HH_1$. Its Griffiths negativity proved above implies that $G_{h_1}$ is strictly plurisubharmonic off the zero section. The first-jet identity \eqref{eq:jetintro} therefore gives
\[
 \tr(B_1\Psi_1(q))=f_{h_1}(q)>0\qquad(q\in\KK),
\]
which proves Example~\ref{ex:levelone}\textup{(\ref{ex:positivelevel})}. Theorem~\ref{thm:A} applied to this matrix yields Griffiths positivity of $E$. At level zero, the evaluation map $\widehat\Phi:E^*\setminus0\to H^0(X,E)^*$, given by $\widehat\Phi(x,v)(s)=v(s(x))$, has a nonzero differential kernel; this gives \eqref{eq:zeroatom} and rules out every Hermitian matrix $B$ at that level.

\section*{Acknowledgements}
S.-Y. Xie acknowledges partial support from the National Key R\&D Program of China under Grants No.~2023YFA1010500 and No.~2021YFA1003100, and from the National Natural Science Foundation of China under Grants No.~12288201 and No.~12471081, as well as support from the Xiaomi Young Talents Program.

\section*{AI Use Disclosure}
The authors developed the general strategy of this paper. The explicit example was found with AI assistance, and AI tools were also used to refine the language and presentation. The authors take full responsibility for the content and correctness of the paper.

\end{document}